\documentclass[10pt,a4paper]{article}
\usepackage[utf8]{inputenc}
\usepackage{geometry}
\usepackage{graphicx}
\usepackage{amssymb}
\usepackage{epstopdf}
\usepackage{mathrsfs}
\usepackage{amsmath}
\usepackage{amsthm}

\newcommand\R{\mathbb{R}}

\newcommand\C{\mathbb{C}}

\begin{document}

\newtheorem{theoreme}{Theorem}
\newtheorem{definition}{Definition}
\newtheorem{lemme}{Lemma}
\newtheorem{remarque}{Remark}
\newtheorem{exemple}{Example}
\newtheorem{proposition}{Proposition}
\newtheorem{corolaire}{Corollary}
\newtheorem{hyp}{Hypothesis}
\newtheorem*{theo}{Theorem}

\title{Sharp resolvent bounds and resonance-free regions}
\author{Maxime Ingremeau}

\maketitle
\begin{abstract}
In this note, we consider semiclassical scattering on a manifold which is Euclidean near infinity or asymptotically hyperbolic. We show that, if the cut-off resolvent satisfies polynomial estimates in a strip of size $O(h |\log h|^{-\alpha})$ below the real axis, for some $\alpha\geq 0$, then the cut-off resolvent is actually bounded by $O(|\log h|^{\alpha+1} h^{-1})$ in this strip. As an application, we improve slightly the estimates on the real axis given in \cite{bourgain2016spectral} in the case of convex co-compact surfaces.
\end{abstract}
\section{Introduction}
Let $(X,g)$ be a Riemannian manifold which is either Euclidean near infinity or which is asymptotically hyperbolic and even. Let $V\in C_c^\infty(X)$, an consider the $h$-dependent family of operators
$$P_h = -h^2\Delta_g +V,$$
and the family of its outgoing resolvent $R_+(z;h) = (P_h-z)^{-1}$, which is well defined for $\Im(z)>0$.

It is well-known (see \cite[\S 4 and \S5]{Resonances}) that, if $\chi \in C_c^\infty(X)$, then for any $h>0$, $z \mapsto \chi R_+(z;h)\chi$ can be extended to $\C\backslash (-\infty,0]$ as a meromorphic function. Its poles, which are independent of the choice of $\chi$, are called the \emph{resonances} of $P_h$.

\begin{theoreme}\label{th}
Let $(X,g)$ and $P_h$ be as above. Fix $E_0>0$ and $\chi \in C_c^\infty(X)$. 

Suppose that there exists $\alpha_1\geq 0$ and $\varepsilon_0, h_0, \alpha_2, C_1, C_2>0$, such that for all $0<h<h_0$, $P_h$ has no resonances in
\begin{equation}\label{eq: def D}
\mathcal{D}_h:=\Big{\{}z\in \C ; \Re z \in [E_0-\varepsilon_0, E_0+\varepsilon_0] \text{ and } \Im z \geq -C_1 \frac{h}{|\log h|^{\alpha_1}}\Big{\}}.
\end{equation}
and such that for all $z\in \mathcal{D}_h$,
\begin{equation}\label{eq:aprioribound}
\|\chi R_+(z;h)\chi \|_{L^2\mapsto L^2} \leq C_2 h^{-\alpha_2}.
\end{equation}

Then there exists $C_0>0$ such that for all $0<h<h_0$ and for all $E\in [E_0-\varepsilon, E_0+\varepsilon]$, we have
\begin{equation}\label{eq:optimal bound}
\|\chi R_+(E;h)\chi \|_{L^2\mapsto L^2} \leq C_0 \frac{|\log h|^{\alpha_1+1}}{h}.
\end{equation}
\end{theoreme}

Note that a converse to this statement was proved in \cite{datchev2012extending}, using ideas from Vodev (See for instance \cite{vodev2014semi}. One may also see \cite[Theorem 6.25]{Resonances}, and the references following the proof of the theorem). In particular, they show that if (\ref{eq:optimal bound}), then similar estimates hold in a strip of size $O(h |\log h|^{-\alpha_1-1})$ below the real axis.

As an application of Theorem \ref{th}, we can improve slightly the bounds on the resolvent given in \cite{bourgain2016spectral} in the case of convex co-compact hyperbolic surfaces. Indeed, Theorem 2 in \cite{bourgain2016spectral} implies that the point (ii) of our theorem is satisfied with $\alpha_1=0$. Therefore, we obtain
\begin{corolaire}
Let $(X,g)$ be a convex co-compact hyperbolic surface and let $\chi\in C_c^\infty(X)$.  Then there exists $C, C', h_0>0$ such that for any $\epsilon>0$ and any $0<h<h_0$, we have
\begin{equation*}
\forall z \in B\Big{(}1, C \frac{h}{|\log h|}\Big{)},~~ \|\chi (-h^2\Delta_g - z)^{-1}\chi\|_{L^2\rightarrow L^2}\leq C \frac{|\log h|}{h},
\end{equation*}
where $(-h^2\Delta_g - z)^{-1}$ denotes the outgoing resolvent of $-h^2\Delta_g$.
\end{corolaire}

The bound (\ref{eq:optimal bound}) with $\alpha_1=0$ is known to be optimal when the dynamics has a non-empty trapped set at energy $E$, as was shown in \cite{bonyburqramondminoration}. It is known to hold in several situations where the dynamics is \emph{hyperbolic} near the trapped set at energy $E$: see for instance \cite{burq2004smoothing}, \cite{christianson2007semiclassical}, \cite{christianson2008dispersive},\cite{NZ}, \cite{nonnenmacher2009semiclassical} and \cite{wunsch2010resolvent} and the references therein.

In \cite{christianson2013local}, the authors consider manifolds $(X,g)$ with a single trapped trajectory, which is hyperbolic in a degenerate way. On such manifolds, they show that $\|\chi R_+(E;h)\chi \|_{L^2\mapsto L^2} \leq C h^{-\alpha}$ for some $\alpha>1$, but that such an estimate is false for any constant $\alpha'<\alpha$. Therefore, by the result of \cite{datchev2012extending}, the resolvent is polynomially bounded in a strip of size $h^{-\alpha}$ below the real axis, but it does not satisfy (\ref{eq:optimal bound}). This shows that our result does not hold if, in $(ii)$, we replace $\mathcal{D}_h$ by $\mathcal{D}''_h:=\Big{\{}z\in \C ; \Re z \in [E_0-\varepsilon_0, E_0+\varepsilon_0] \text{ and } \Im z \geq -C h^{-\alpha}\Big{\}}$ for some $\alpha>1$.

\paragraph{Acknowledgments}
The author would like to thank Stéphane Nonnenmacher for useful discussion during the writing of this note.

The author is partially supported by the Agence Nationale de la Recherche project GeRaSic (ANR-13-BS01-0007-01).

\section{Proof of Theorem \ref{th}}\label{sec : proof}
Fix constants $\varepsilon_0, h_0, \alpha_1, \alpha_2, C_1, C_2>0$ as in the statement of $(ii)$, and  we fix a function $\psi\in C^\infty_c(E_0-\varepsilon_0, E_0+\varepsilon_0)$ such that $\psi(x) = 1$ for $x\in (E_0-\varepsilon_0/2, E_0+\varepsilon_0/2)$.

If $t\in \R$, we shall write $U_h(t)$ or $e^{i\frac{t}{h} P_h}$ for the Schrödinger propagator of $P_h$.
Our first lemma says that the truncated propagators become small after times of the order of a large constant times $|\log h|$.

\begin{lemme}
For any $r>0$, there exists $M_r>0$ and $C_r>0$ such that 
$$\|\chi U_h(M_r |\log h|^{\alpha_1+1})  \psi(-h^2 \Delta) \chi\|_{L^2\rightarrow L^2} \leq C_r h^r.$$  
\end{lemme}
The proof is very similar to that of \cite[Theorem 7.15]{Resonances}
\begin{proof}
Let us consider the incoming resolvent $R_-(z;h) := (P_h- z)^{-1}$, which is analytic for $-\Re z >0$.
Using Stone's formula, we obtain that for any $t>0$, we have
\begin{equation*}
\chi U_h(t)  \psi(-h^2 \Delta) \chi = \frac{1}{2i \pi} \int_{\R} e^{-it z/h} \chi \big{(}R_-(z;h) -  R_+(z;h)\big{)} \psi(z) \chi \mathrm{d}z.
\end{equation*}

Let $\tilde{\psi}$ be an almost analytic extension of $\psi$, that is to say, a function $\tilde{\psi}\in C_c^\infty(\C)$ such that 
\begin{equation}\label{eq: almost analytic}
\partial_{\overline{z}} \psi (z) = O\big{(}(\Im z)^\infty\big{)}
\end{equation}
 and such that $\tilde{\psi}(z)= \psi(z)$ for $z\in \R$. We may furthermore assume that $$\mathrm{spt }~ \tilde{\psi}\subset \{z; \Re z \in \mathrm{spt}~\psi\}.$$
We refer the reader to \cite[\S 2]{martinez2002introduction} for the construction of such a function.

Using Green's formula, we obtain that 
\begin{equation*}
\begin{aligned}
\chi U_h(t)  \psi(-h^2 \Delta) \chi &= \frac{1}{2i \pi} \int_{\Im z = -C_1 h|\log h|^{-\alpha_1}} e^{-it z/h} \chi \big{(}R_-(z;h) -  R_+(z;h)\big{)} \tilde{\psi}(z) \chi \mathrm{d}z\\
&+ \frac{1}{2i \pi} \int_{-C_1 h|\log h|^{-\alpha_1}\leq \Im z \leq 0} e^{-it z/h} \chi \big{(}R_-(z;h) -  R_+(z;h)\big{)} \partial_{\overline{z}} \psi(z) \chi \mathrm{d}z.
\end{aligned}
\end{equation*}

Thanks to (\ref{eq:aprioribound}) and to (\ref{eq: almost analytic}), the second term is $O(h^\infty)$, independently of $t$. On the other hand, by (\ref{eq:aprioribound}), the first term is bounded by $C e^{-C_1 t |\log h|^{-\alpha_1}} h^{-\alpha_2}$. Therefore, taking $t= M |\log h|^{\alpha_1+1}$ with $M$ large enough proves the lemma.
\end{proof}

The rest of the proof is similar to \cite[\S 9]{NZ}.
In the following two lemmas, we use our propagator estimates to deduce bounds on the outgoing resolvent when $\Im z>0$.

\begin{lemme}
For all $r>0$, there exists $C'_r>0$ such that for all $z\in \C$ with $\Im z>0$, we have 
$$\|\chi (P_h -z)^{-1}\psi(-h^2 \Delta)\chi \|_{L^2\rightarrow L^2} \leq \frac{C_r |\log h|^{\alpha_1+1}}{h} + C_r \frac{h^r}{\Im z}.$$
\end{lemme}
\begin{proof}
Since $\Im z>0$, we have
\[
\chi (P_h -z)^{-1}\psi(-h^2 \Delta)\chi = \frac{1}{h}\int_0^\infty e^{itz/h} \chi  e^{-i\frac{t}{h} P_h}\psi(-h^2 \Delta)\chi \mathrm{d}t,
\]
so that, 
\[
\begin{aligned}
\|\chi (P_h -z)^{-1}\psi(-h^2 \Delta)\chi \|_{L^2\rightarrow L^2} &= \frac{1}{h}\int_0^\infty e^{- \frac{t}{h} \Im z} \| \chi e^{-i\frac{t}{h} P_h}\psi(-h^2 \Delta)\chi\| \mathrm{d}t\\
&= \frac{1}{h}\int_0^{M_r|\log h|^{\alpha_1+1}} e^{- \frac{t}{h} \Im z} \| \chi e^{-i\frac{t}{h} P_h}\psi(-h^2 \Delta)\chi\| \mathrm{d}t\\
&+ \frac{1}{h}\int_{M_r|\log h|^{\alpha_1+1}}^\infty e^{- \frac{t}{h} \Im z} \| \chi e^{-i\frac{t}{h} P_h}\psi(-h^2 \Delta)\chi\| \mathrm{d}t\\
&\leq  \frac{1}{h}\int_0^{M_r|\log h|^{\alpha_1+1}} \mathrm{d}t + \frac{1}{h}\int_{M_r |\log h|}^\infty e^{- \frac{t}{h} \Im z} C_r h^r \mathrm{d}t \\
&\leq M_r \frac{|\log h|^{\alpha_1+1}}{h} + C_r h^{r-1} \int_0^\infty  e^{- \frac{t}{h} \Im z}\mathrm{d}t\\
&\leq M_r \frac{|\log h|^{\alpha_1+1}}{h} + C_r \frac{h^{r}}{\Im z}.
\end{aligned}
\]

 This proves the lemma.
\end{proof}

\begin{lemme}\label{lem: sanscutoff}
For all $r>0$, there exists $C_r>0$ such that for all $z\in \C$ such that $\Re z\in [E_0-\varepsilon_0/8, E_0+\varepsilon_0/8]$, $\Im z>0$, we have 
$$\|\chi (P_h-z)^{-1}\chi \|_{L^2\rightarrow L^2} \leq \frac{C_r |\log h|^{\alpha_1+1}}{h} + C_r\frac{h^r}{\Im z}.$$
\end{lemme}
\begin{proof}
By the preceding lemma, we only have to show that
\begin{equation}\label{eq: away}
\|\chi(P_h -z)^{-1}\big{(}Id-\psi(-h^2 \Delta)\chi \big{)}\|_{L^2\rightarrow L^2} \leq \frac{C_r |\log h|^{\alpha_1+1}}{h} + \frac{h^r}{\Im z}.
\end{equation}

The proof is standard, and similar to \cite[Lemma 9.1]{NZ} or \cite[Theorem 6.4]{Zworski_2012}, but we recall the main lines for the reader's convenience.

Let us denote the symbol of $P_h$ by 
\begin{equation}\label{eq:def p}
p : T^*X\ni (x,\xi)\mapsto |\xi|_g^2 +V(x)\in \R.
\end{equation} 

Consider a function $\psi_1\in C^{\infty}(T^*X)$ such that $\psi_1(x,\xi) = 1$ when $p(x,\xi)\in (E_0-\varepsilon_0/4, E_0+\varepsilon_0/4)$ and  $\psi_1(x,\xi) = 0$ when $p(x,\xi)\notin (E_0-\varepsilon_0/3, E_0+\varepsilon_0/3)$. We shall denote by $\psi_1^w$ the Weyl quantization of $\psi_1$ acting on $L^2(X)$,  as defined in \cite{Zworski_2012}.
One can show that for all $z\in \C$ such that $\Re z\in [E_0-\varepsilon_0/8, E_0+\varepsilon_0/8]$ and $\Im z>0$, we have that $(P_h-z+ i \psi_1^w)^{-1}:L^2(X)\rightarrow L^2(X)$ exists and is a pseudo-differential operator bounded independently of $h$.

Furthermore, we have $\big{\|}\psi_1^w (Id-\psi(P_h))\big{\|}_{L^2\rightarrow L^2} = O(h^\infty)$, so that we have that for any $f\in L^2(X)$
\[
(P_h-z) (P_h-z+ i \psi_1^w)^{-1} \big{(}Id-\psi(P_h)\big{)} f = (1-\psi(P_h)) f +R,
\]
with $\|R\|_{L^2}=O(h^\infty)$. Therefore, $(P_h-z+ i \psi_1^w)^{-1} \big{(}Id-\psi(P_h)\big{)} f$ is an approximate inverse of $(1-\psi(P_h)) f$ by $P_h-z$, and it is bounded independently of $h$, so that (\ref{eq: away}) holds.
\end{proof}

Let us now conclude the proof of Theorem \ref{th}.
\begin{proof}[Proof of Theorem \ref{th}]
Let us fix a $E\in [E_0-\varepsilon_0/20, E_0+\varepsilon_0/20]$. Let $f,g\in L^2(X)$. Consider the map $u_{f,g} : z\mapsto e^{-\frac{(E-z)^2}{h^{2}}} \langle f, (P_h-z)^{-1} g \rangle$. The map $u_{f,g}$ is holomorphic for $ \Im z \geq-C_1 h|\log h|^{-\alpha_1}$, and for each $y\in [-C_1 h|\log h|^{-\alpha_1},h]$, we have $$\sup_{x\in \R} |u_{f,g}(x+iy)| = \sup_{x\in (E_0-\varepsilon_0/8,E_0+\varepsilon_0/8)} |u_{f,g}(x+iy)|.$$
Let us write $$v(y) = \sup_{x\in \R} |u_{f,g}(x+iy)|.$$

By assumption, we have $$v(-C_1 h|\log h|^{-\alpha_1})\leq C_2 h^{\alpha_2},$$ while by Lemma \ref{lem: sanscutoff}, we have that, for all $N>0$, there exists $C_N>0$ such that $ |v (h^N)|\leq C_N \frac{|\log h|^{\alpha_1+1}}{h}.$
We now use Hadamard's three lines theorem, which tells us that $[-C_1 h|\log h|^{-\alpha_1}, \infty) \ni y\mapsto \log(v(y))$ is a convex function. In particular, we have
\[ \begin{aligned}
\log u_{f,g}(1)\leq  \log v(0)&\leq \log v(h^N) + \frac{h^N}{C_1 h|\log h|^{-\alpha_1} +h^N} \log (v(-C_1 h|\log h|^{-\alpha_1})- v(h^N))\\
&\leq \log \big{(}C_N \frac{|\log h|^{\alpha_1+1}}{h}\big{)} + c h^{N-1} |\log h|\\
& = C' \log \big{(}|\log h|^{\alpha_1+1}\big{)} + |\log h| + o_{h\rightarrow 0}(1).
\end{aligned}
\]

Therefore, we obtain that
$$\langle f, (P_h -E)^{-1} g \rangle \leq C \frac{|\log h|^{\alpha_1+1}}{h}.$$
Since this is true for all $f,g$, we deduce (\ref{eq:optimal bound}). This concludes the proof of the theorem.
\end{proof}

\bibliographystyle{alpha}
\bibliography{references}
\end{document}